\documentclass[a4paper,10pt]{article}
\usepackage[utf8x]{inputenc}
\usepackage{amsmath,amssymb,amsthm}

\newtheorem{theorem}{Theorem}[section]
\newtheorem{proposition}[theorem]{Proposition}
\newtheorem{corollary}[theorem]{Corollary}
\newtheorem{definition}[theorem]{Definition}
\newtheorem{assumption}[theorem]{Assumption}

\newtheorem{lemma}[theorem]{Lemma}

\newcommand{\Fps}{{F'}^*}

\newcommand{\T}{\nabla J}
\newcommand{\xs}{x^*}
\newcommand{\Bx}{B_\rho(\xs)}
\newcommand{\half}{\frac{1}{2}}

\newcommand{\R}{\mathbb{R}}
\newcommand{\N}{\mathbb{N}}
\newcommand{\esub}[1]{e_{k_#1}}

\newcommand{\xd}{x^{\delta}}
\newcommand{\ed}{e^{\delta}}
\newcommand{\yd}{y^\delta}
\newcommand{\Jd}{{J^\delta}} 
\newcommand{\Td}{\nabla J^\delta}
\newcommand{\Hd}{H^\delta}

\newcommand{\scl}{\langle} 
\newcommand{\scr}{\rangle}
\newcommand{\Cm}{\theta}
\newcommand{\myeta}{\kappa}
\newcommand{\co}{c_0}

\newcommand{\Nd}{N_{\delta}}
\newcommand{\xdl}{x^{\delta_l}}
\newcommand{\Ndl}{N_{\delta_l}}
\newcommand{\Jls}{J_{LS}}
\newcommand{\Ld}{L_\delta}
\newcommand{\FL}{\phi}

\newcommand{\fack}{\xi}
\newcommand{\Nc}[1]{{\rm NC(}#1{\rm)}}

\newcommand{\skipp}[1]{}

\title{Convergence of the gradient method for ill-posed problems} 
\author{Stefan Kindermann}
\date{}
\begin{document}
\maketitle 
\begin{abstract}
We study the convergence of the gradient descent method for solving 
 ill-posed problems where the solution is characterized
as  a global minimum of a differentiable functional in a Hilbert space. 
The classical least-squares functional for nonlinear operator equations is 
a special instance of this framework and the gradient method then reduces to 
Landweber iteration. The main result of this article is a  
proof of weak and strong convergence under new nonlinearity conditions that
generalize the classical tangential cone conditions. 
\end{abstract}

\section{Introduction}
A widely-used approach for dealing with a nonlinear ill-posed problem is to phrase 
it as an operator equation in Banach- or Hilbert spaces and apply  an iterative 
regularization method for its solution \cite{EHN96}. The simplest, though not the fastest,
amongst them is the  Landweber iteration, which can be viewed as a gradient descent
method for the associated least-squares functional. A well-known 
convergence theory has been established for Landweber iteration for nonlinear 
ill-posed problems based on the seminal paper by Hanke, Neubauer, and Scherzer \cite{HaNeSc95}. 
The pivotal innovation that paved the way for the analysis is to include appropriate 
restrictions on the ``nonlinearity'' of the problem by imposing so-called nonlinearity 
conditions on the underlying operators. 


Such conditions have been verified for several important nonlinear ill-posed problems,
e.g., for parameter identification in partial differential equations using interior measurements; see, e.g.,  \cite{KaNeSc08}.
However, 
and this is the crucial point, they have not yet been verified for certain well-studied 
problems like the electrical impedance tomography (aka.~Calder{\'o}n's  problem) \cite{Cal80}---although 
Landweber and  other iterative methods have been successfully applied to them. 
 This might give a hint that the traditionally used nonlinearity conditions are too strong 
to be satisfied for certain applications, and one may try to replaced them by weaker assumptions.

The main goal of this paper is to prove (local) weak and strong convergence of a subsequence 
of the gradient descent iterates for a functional with  Lipschitz-continuous gradient imposing 
more general nonlinearity conditions than the usual ones. 

Reviewing such typically-used restrictions reveals  
that the most common ones  
are the weak and strong form of the tangential cone conditions 
\cite{HaNeSc95,Sch95,Va98}. Stronger than those are the range-invariant conditions \cite{HaNeSc95,Ne00}.
Conceptually similar to this is the approach via Hilbert scales \cite{Ne00,Ne15}.
A typical functional-based nonlinearity condition is the assumption that the underlying
functional is locally a convex one, or equivalently formulated as the gradient being 
monotone \cite{BaKuSm11,BaGo94,Ra05,HoRa09,HoRa10} or strongly monotone \cite{PePiSe06,BaSm06}. 
Except for strong monotonicity, the assumption of a monotone gradient is not enough to 
prove strong convergence for the classical Landweber iteration. (Note that in  \cite{Ra05,HoRa09,HoRa10}
a continuous version was investigated while in  \cite{BaSm06,BaKuSm11,BaGo94} a modified (i.e., regularized) 
Landweber iteration was 
considered.) An insightful comparison of tangential cone conditions and several versions of monotonicity
of the gradient can be found in  \cite{Sch95}.

We note that such conditions are also used for proving convergence of other iterative regularization
methods in the nonlinear case such as, e.g., Gauss-Newton-type iterations \cite{Ka97,BaSm06,BaSm07,Jin11},
the Levenberg-Marquardt scheme \cite{Ha97}, or Kaczmarz iterations \cite{HaLeSc07}.


One of the main contribution in this article is to   prove boundedness
and weak convergence of the gradient descent iterations essentially 
under a two-parametric nonlinearity condition, which generalizes
 and includes both the weak/strong tangential conditions and several convexity conditions as 
special cases. This is interesting insofar as the tangential cone conditions do not imply convexity 
of the associated least-squares functional, thus our analysis can be viewed as an attempt for a unification 
of the established nonlinearity restrictions.

We also prove strong convergence of the iterates 
under a novel restriction which requires the functional to be ``balanced'' around 
critical points. This can be seen as a generalization of the strong tangential cone condition.
All these results hold both for the exact data-case and the noisy data-case, 
where in the latter, we employ a simple a-priori parameter choice. 

 Our setup is phrased as that of  
the problem of finding minina of general ill-posed functional rather than in the form of  nonlinear operator equations,
but since, of course, one can apply the least-squares idea,  the classical Landweber 
iteration is a special instance of the gradient iteration studied here. 
Note that the Lipschitz-continuity of the gradient 
plays an essential role in our work, hence, certain Banach-space variants of Landweber iterations  (see, e.g., \cite{HeKa10,ScKaHo12})
in non-smooth spaces are not within the scope of this work.

Our paper is organized as follows: in Section~\ref{Sec:two} we define the gradient iteration, 
present the standard assumptions we use and the novel nonlinearity conditions we impose. We study them  in detail 
by relating them to the traditionally used ones. In Section~\ref{Sec:three}, we 
prove boundedness and weak convergence of the iterates to a stationary point for our setup  both in 
the case of exact and noisy data. In Section~\ref{Sec:four},  strong convergence of the
iteration is proven in a similar framework.

\section{Setup and nonlinearity conditions}\label{Sec:two}
We consider the problem of finding a solution of 
an ill-posed problem that is characterized as a global minimum of a certain 
functional. Throughout this paper, we denote by $\Bx$ a ball with 
center $\xs$ and radius $\rho$ in a Hilbert space. 
We assume given a 
Fr\'{e}chet-differentiable, nonnegative  functional 
\begin{equation}\label{fun}
J: \Bx \subset X \to \R^+, 
\end{equation}
where $X$ is a Hilbert space, 
and that a sought-for solution $\xs$ satisfies 
\begin{equation}\label{opc}
J(\xs) = 0.
\end{equation}
By the nonnegativity of $J$, $\xs$ is a global minimum and has to satisfy the first-order 
optimality condition 
\begin{equation}\label{firstorder}
\nabla J(\xs)  = 0. 
 \end{equation}
The most important instance of such a functional is the 
least-squares functional $\Jls(x)$ for a nonlinear operator equation  with given data $y$,
 \begin{equation}\label{main}
          F(x) = y, 
\end{equation}
which is defined as 
\begin{equation}\label{least squares}
\Jls(x) = \tfrac{1}{2} \|F(x) - y\|^2.
\end{equation}
In the setup of \eqref{fun}, we assumed that the given data are 
encoded somehow into the functional $J$. Similar to the 
least-squares case, we have to allow for inexact data as well, 
i.e., we have to consider a ``noisy'' version of $J$ that represents the 
actual measurements. 

Thus, we assume given a 
Fr\'{e}chet-differentiable, nonnegative  functional 
\begin{equation}\label{fund}
\Jd: \Bx \subset X \to \R^+, 
\end{equation}
where the actual iteration is based upon. 
In order to solve \eqref{opc} for $\xs$ with given noisy data,
a gradient iteration can be used. It is defined iteratively 
(as long as the iterates $\xd_k$ stay in $\Bx$) as 
\begin{equation}\label{lwnoisy}
\xd_{k+1} = \xd_k - \nabla \Jd(\xd_k), \qquad k = 0,\ldots, 
\end{equation}
starting with an initial guess $\xd_0 \in \Bx$.  
For the analysis, it is convenient to define the corresponding iteration with 
exact data as well,
\begin{equation}\label{lwexact}
x_{k+1} = x_k - \nabla J(x_k) , \qquad k = 0,\ldots, 
\end{equation}
starting with the same initial guess $x_0 \in \Bx$ as in \eqref{lwnoisy}.
In the least-squares case \eqref{least squares}, iteration \eqref{lwnoisy}, respectively \eqref{lwexact}, is
the classical Landweber iteration: 
\begin{equation}
\xd_{k+1} = \xd_k -  \Fps(\xd_k)\left(F(\xd_k) - \yd \right), \qquad k = 0,\ldots 
\end{equation}
Note that usually the gradient descent iterations use a stepsize parameter in front of 
the gradient term. We assume throughout a constant stepsize  parameter that is 
encoded into the functional $J$ such that it will  be set to $1$ throughout.  The only restriction 
on the stepsize comes from the assumptions that we impose on $J$ and $\Jd$.  
Essentially, we assume that $J$ and $\Jd$ are differentiable on $\Bx$ with 
Lipschitz-continuous derivative and Lipschitz constant smaller than $1$. Precisely, we postulate the following:

\begin{assumption}\label{Assmain}\ \\[-4mm] \nopagebreak
\begin{enumerate}
\item $X$ is a Hilbert space. 
\item There exists an $\xs$ (exact solution) which satisfies \eqref{opc} (and hence also \eqref{firstorder}). 
\item For some $\rho>0$,  $J$ and $\Jd$ are defined on $\Bx \subset X$ and are Fr\'{e}chet-differentiable there. 
\item\label{A2k}  For all $x \in \Bx,$  the gradient $\nabla J(x)$ 
is Lipschitz continuous with  Lipschitz constant  $ L   < 1 $:
\[ \|\nabla J(x_1+x_2) - \nabla J(x_1) \| \leq L \|x_2\|  < \|x_2\|, \qquad \forall x_1 , x_1+x_2\in \Bx. \]
\item The functional $\Jd$  satisfies 
\begin{equation}\label{five}
  \|\nabla \Jd(x) - \nabla J(x) \| \leq \delta  \qquad \forall x \in \Bx,    \qquad \mbox{and} 
\end{equation} 
\begin{equation}\label{Jfive}
  \left| \Jd(x) -  J(x) \right|  \leq \psi(\delta)  \qquad \forall x \in \Bx, \quad \lim_{\delta \to 0} \psi(\delta) = 0\, .
\end{equation} 
\item The gradient satisfies 
\begin{equation}\label{Fleq} \|\nabla \Jd(x)\|^2 \leq  \FL(\Jd(x)) \qquad \forall x \in \Bx \end{equation}
with some monotone positive continuous function $\FL$ with $\FL(0) = 0$.  
%
\end{enumerate}
\end{assumption} 
As the notation suggests, $\delta$ plays the role of the noise level. We note that 
\eqref{five} implies that $\nabla \Jd$ is Lipschitz continuous with Lipschitz constant $\Ld$
\begin{equation}\label{noise1} \Ld \leq L + \delta. \end{equation}

It is easy to observe that for the least-squares problem \eqref{least squares} and Landweber iteration,  these 
assumptions are satisfied if $F$ has a Lipschitz-continuous derivative and if the stepsize in Landweber iteration 
is chosen sufficiently small. The noise level $\delta$ according to \eqref{noise1} is then related to the usual one by
$\delta \geq \sup_{x \in \Bx} \|F'(x)\| \|\yd -y\|.$  Since for the least-squares case, we have 
$\nabla \Jd(x) = F'(x)^*(F(x) - \yd)$,  the inequality \eqref{Fleq} holds with $\FL(s) \sim s.$

\subsection{Nonlinearity conditions}
We now propose a two-parametric generalization  of the well-known weak tangential cone condition: 
\begin{definition}\label{thisdef}
For some $\gamma \in [0,\infty]$ and $\beta \in \R$, we say that \Nc{$\gamma$,$\beta$} is satisfied for $J$ 
if for all $x_1,x_2 \ in \Bx$ the following implication holds true:
\begin{equation}\label{weakgamma}
J(x_1) \leq \gamma J(x_2) \Rightarrow 
\scl \nabla J(x_2),x_2-x_1 \scr \geq - \beta \|\nabla J(x_2)\|^2. 
\end{equation} 
\end{definition}
Note that we allow $\gamma = 0$ and $\gamma = \infty$. In the later case, the premise in the implication is 
tautological, thus the conclusion has to hold for all $x_1,x_2 \in \Bx$, while for $\gamma = 0$, the 
conclusion has to hold only for $x_1$ at a global minimum.

It is easy to verify that the condition in Definition~\ref{thisdef} is the stronger the larger the $\gamma$ and 
the smaller the $\beta$ is:
\begin{align*} \text{ for } \gamma_1 \leq \gamma_2 : & \quad \text{\Nc{$\gamma_2$,$\beta$}} \Rightarrow  
\text{\Nc{$\gamma_1$,$\beta$}} \\
 \text{ for } \beta_1 \leq \beta_2 : & \quad  \text{\Nc{$\gamma$,$\beta_1$}} \Rightarrow  
\text{\Nc{$\gamma$,$\beta_2$}} \, .
\end{align*}
%

This condition can be compared to 
the weak tangential cone condition 
(or  $\xs$-quasi-uniform monotonicity  \cite{Sch95}) in the least-squares case:
there exists an $0 < \eta < 1$ such that  
\begin{equation}\label{wsc}
\scl F(x) - F(\xs) - F'(x)(x-\xs), F(x) - F(\xs) \scr  \leq \eta \|F(x) - F(\xs) \|^2  \quad  \forall x \in \Bx. 
\end{equation} 
or, equivalently, 
\begin{equation}\label{wscm}
\scl F'(x)(x-\xs), F(x) - F(\xs) \scr  \geq (1-\eta) \|F(x) - F(\xs) \|^2  \quad  \forall x \in \Bx. 
\end{equation} 
It is easy to see that for Fr\'{e}chet-differentiable $F$, \eqref{wsc} with $\eta \in (0,1)$ implies \eqref{weakgamma} with a 
{\em negative} $\beta$ and $\gamma = 0$ 
for the associated least-squares functional. 
It was shown  \cite{Sch95} that \eqref{wsc} with $\eta \in (0,1)$ implies weak convergence 
for the Landweber iteration with exact data.  
In \cite{Va98},  the condition \Nc{0,$\beta$} with $\beta <0$ was imposed and again weak convergence 
of the exact Landweber iteration was proven (and also strong convergence for a modified form of the iteration). 

We generalize these results insofar as we also verify weak convergence 
in the noisy case and more interesting, we prove that \eqref{weakgamma} with $\gamma=0$ 
and {\em any} $\beta \in \R$ (also positive ones! and in particular with \eqref{wsc} with $\eta = 1$) 
already implies weak convergence of a subsequence of the 
gradient iteration.

Strong convergence of gradient iterations require a stronger nonlinearity condition than the previous ones. 
For our  convergence analysis, we need  additionally to \eqref{weakgamma} the following, which we denote 
as  balancing condition.
\begin{definition}\label{gbal}
Let $\gamma\geq 0$. 
We say that the functional $J:\Bx \to \R^+$ is $\gamma$-balanced around $\xs$ if for some $\rho_0< \rho$ 
and any sequence $z_n$ with $\rho_0 \leq \|z_n\| \leq \rho$ there exists 
a $ \tau >0$ and a $n_0 \in \N$ such that 
%
\begin{equation}\label{fancy1} J(\xs - \tau z_n ) \leq  \gamma J(\xs+ z_n) \qquad \forall n \geq n_0\, . 
\end{equation}
\end{definition}
We will prove strong convergence of the gradient iteration under the condition that 
$J$ is $\gamma$-balanced and satisfies \Nc{$\gamma$,$\beta$} for some $\gamma \geq 0$ and
some $\beta \in \R$.

The condition in Definition~\ref{gbal} can sloppily be interpreted as the requirement that $J$ does not have extremely 
large values 
when evaluated at a mirror point around $\xs$.
Thus, the  functional should roughly behave in a similar 
way left and right at $\xs$  on a line through $\xs$.

It is easy to verify that if $J$ is convex on $\Bx$ and satisfies a symmetry condition 
\[ J(\xs -z) \leq C J(\xs +z) \qquad \forall z \in \Bx, \]
with a constant $C$, then \eqref{fancy1} holds.


Traditionally, strong convergence of  the Landweber iteration is verified under  the so-called (strong)
 tangential cone condition (or strong Scherzer condition) (see, e.g., \cite{HaNeSc95,EHN96,KaNeSc08,Sch95}): there exists  $0 < \eta < 1$ such that  
\begin{equation}\label{ssc}
\| F(x) - F(\tilde{x}) - F'(x)(x-\tilde{x}) \| \leq \eta \|F(x) - F(\tilde{x}) \|  \quad  \forall x,\tilde{x} \in \Bx. 
\end{equation} 
It is obvious that the strong tangential cone condition implies the weak one. There are several interesting 
conclusions that follow from \eqref{ssc}. For instance, the following useful estimate 
follows immediately from \eqref{ssc}:
\begin{equation}\label{furthermore}
\frac{1}{1+\eta} \|F'(x)(\tilde{x} -x)\| \leq \|F(\tilde{x}) - F(x) \| \leq  \frac{1}{1- \eta} \|F'(x)(\tilde{x} -x)\|  \quad  \forall x,\tilde{x} \in \Bx. 
\end{equation}

Instead of \eqref{ssc}  an even stronger condition, is sometimes imposed: it postulates 
the existence of a family of operators $R_x$ such that 
\[ F'(x)  = R_x F'(\xs)  \qquad \forall x \in \Bx \text{ and }  \|R_x - I\| \leq C \|x-\xs\|. \]
Locally, it follows that $R_x$ are invertible  operators, which allows to compare the derivatives at 
different points $x$. It follows that this condition implies \eqref{ssc} (possibly on a smaller ball). 
%

Let us briefly study the relations of conditions \Nc{$\gamma$,$\beta$} and Definition~\ref{gbal} with the both the tangential 
cone conditions and certain convexity conditions. At first we introduce a generalization 
of quasiconvexity:
\begin{definition}
Let $C$ be a convex set and  let $0 \leq \gamma \leq 1.$  
We say that the functional $f$ is $\gamma$-quasiconvex if 
for all $x,y \in C$ and $\lambda \in (0,1)$ we have 
\begin{equation}
f(x_1) \leq \gamma f(x_2) \Rightarrow f(\lambda x_1 + (1-\lambda)x_2)  \leq f(x_2).  
\end{equation}
\end{definition}
For $\gamma =1$, we encounter the traditional definition of quasiconvexity \cite{genco}, which might also be phrased as 
the condition 
 that the following inequality holds:
\begin{equation}
 f(\lambda x_1 + (1-\lambda)x_2)  \leq \max\{f(x_2),f(x_1)\}   \qquad \forall x_1,x_2 \in C, \lambda \in (0,1).
\end{equation}
For positive functionals $f$, the assumption of   $\gamma$-quasiconvexity is weaker 
than quasiconvexity, which itself is in any case weaker than convexity. In terms of level-sets,
it is easy to see that $f$ is quasiconvex if and only if all its lower level sets 
$\{f \leq \alpha\}$  are convex. 
We may  view $\gamma$-quasiconvexity as the condition that the 
convex hull of $\{f \leq \gamma \alpha\}$ does not intersect the complement of $\{f \leq \alpha\}$.

The following characterization of $\gamma$-quasiconvexity is useful:
\begin{proposition}
Let $f:X\to \R$ be  Fr\'{e}chet-differentiable on the open convex set $C$. 
Let $0 \leq \gamma \leq 1$. 
Then $f$ is $\gamma$-quasiconvex on $C$ if and only if \Nc{$\gamma$,$0$} holds. 
%
\end{proposition}
\begin{proof}
We follow \cite{arrent}; cf.~\cite[Theorem~3.11]{genco}. If $f$ is $\gamma$-quasiconvex, then for 
any $x_1,x_2$ with $f(x_1) \leq \gamma f(x_2)$, we have 
$\frac{f(\lambda x_1 + (1-\lambda) x_2) -f(x_2)}{\lambda} \leq 0.$ The limit $\lambda \to 0$ implies that 
$\scl \nabla f(x_2), x_1 -x_2 \scr \leq 0$, hence  \Nc{$\gamma$,0} holds. Conversely, suppose that 
 \Nc{$\gamma$,$0$} holds, and define for  $x_1,x_2$  with $f(x_1) \leq \gamma f(x_2)$ the function 
$G(\lambda):= f(\lambda x_1 + (1-\lambda) x_2)$  on the unit interval $[0,1]$.  Suppose that 
$G(\lambda) > G(0)$ for some $\lambda \in [0,1).$ Consider the largest element $\zeta^*$ in the 
nonempty closed set $\{\zeta \in [0,\lambda] | G(\zeta) \leq G(0) \}.$ By construction $\zeta^* < \lambda$ and
it follows by the 
intermediate value theorem that there exists a $\xi \in (\zeta^*,\lambda)$ with $G'(\xi) >0$ and 
$G(\xi) >G(0)$ (since $\zeta^*$ is the largest element in the set).  
However, then with $x_\xi = \xi x_1 + (1-\xi) x_2$ we find that 
$f(x_1) = G(1) \leq \gamma f(x_2) = \gamma G(0) < \gamma f(x_\xi).$ Thus, \Nc{$\gamma$,0} 
implies that 
$$0 \leq \scl \nabla f(x_\xi),x_\xi - x_1 \scr = (1-\xi)\scl \nabla f(x_\xi),x_2 - x_1 \scr = -(1-\xi) G'(\xi),$$
which contradicts $G'(\xi) >0$. Hence, $G(\lambda) \leq  G(0)$ must hold, which implies $\gamma$-quasiconvexity. 
\end{proof}

It is interesting that the weak tangential cone condition can also be expressed by a derivative-free 
condition:
\begin{proposition}\label{thisprop}
Let $F$ be Fr\'{e}chet-differentiable. Then, 
condition \eqref{wsc} holds for some $\eta \in [0,1]$ if and only if the least-squares functional 
$J_{LS}(x):= \tfrac{1}{2} \|F(x) - F(\xs)\|^2$ has the property that the mapping 
\begin{equation}\label{map}
t \to  \frac{1}{t^{2(1-\eta)}}J_{LS}(\xs + t(x-\xs))
\end{equation}
is monotonically increasing for $t \in [0,1]$ for all $x \in \Bx$.
\end{proposition}
\begin{proof}
Let $x_t := \xs + t(x-\xs)$ and let $G(t):=  J_{LS}(\xs + t(x-\xs))$.  
Since 
$$ G'(t) = \scl F'(x_t),(x-\xs),F(x_t) - F(\xs) \scr = \frac{1}{t} \scl F'(x_t),(x_t-\xs),F(x_t) - F(\xs) \scr, $$
we have that \eqref{wsc} implies that 
$G'(t) \geq (1-\eta) 2 \frac{1}{t} G(t),$ from which the monotonicity   of \eqref{map} follows by calculating 
the derivative. On the other hand, if \eqref{map} is monotone then $G'(t) \geq (1-\eta) 2 \frac{1}{t} G(t)$ 
for $t \in (0,1]$  follows easily and hence, taking $t = 1$, we obtain \eqref{wsc}.
\end{proof}

Next, we verify that the strong tangential cone condition implies   \Nc{$\gamma$,$\beta$} with appropriate
parameter values.
 \begin{lemma}
Let the tangential cone condition \eqref{ssc} hold with $\eta <1$. 
 Then the least-squares functional \eqref{least squares} satisfies \Nc{$\gamma$,$\beta$}  with 
$\gamma < \left(\frac{\sqrt{1-\eta^2}}{1 + \sqrt{1-\eta^2}}\right)^2<1$ and
$\beta = - \left[(1-\eta^2) (1 - \sqrt{\gamma})^2  -\gamma\right] \frac{1}{2 (\sup_{x \in \Bx} \|F'(x)^*\|)^2 } <0$.
\end{lemma}
\begin{proof}
By expanding the terms, we verify that \eqref{ssc} is equivalent to the inequality 
\begin{equation}\label{expssc}
  \scl F'(x) (x-z), F(x)-F(z) \scr \geq (1 -\eta^2) \half \|F(x) - F(z) \|^2 + \half \|  F'(x) (x-z)\|^2 
\end{equation}
for all  $x,z \in \Bx$.  
Assume that $J(x_1) \leq \gamma J(x_2)$. Using \eqref{ssc} with $x = x_2$, $z = x_1$, Young's inequality,
and the triangle inequality, we obtain
\allowdisplaybreaks
 \begin{align*} & \scl \nabla J(x_2), (x_2-x_1) \scr =  \scl F(x_2) - F(x_1), F'(x_2) (x_2-x_1) \scr  \\
 &=   \scl F(x_2) - F(x_1), F'(x_2) (x_2-x_1) \scr   +  \scl F(x_1) - F(\xs), F'(x_2) (x_2-x_1) \scr   \\
 &\geq (1-\eta^2) \half \|F(x_2) - F(x_1)\|^2 +  \half \|F'(x_2) (x_2-x_1) \|^2 \\
 & \qquad \qquad -  \|F'(x_2) (x_2-x_1) \| \|F(x_1) - F(\xs)\| \\
 & \geq (1-\eta^2) \half \|F(x_2) - F(x_1)\|^2 - \half \|F(x_1) - F(\xs)\|^2\\
 & \geq (1-\eta^2) \half \left(\|F(x_2) - F(\xs)\|- \|F(x_1) - F(\xs)\| \right)^2   - \half \|F(x_1) - F(\xs)\|^2  \\
 & \geq  (1-\eta^2) (1 - \sqrt{\gamma})^2 \half \|F(x_2) - F(\xs)\|^2 - \gamma  \half \|F(x_2) - F(\xs)\|^2 \\  
 & \geq \left[(1-\eta^2) (1 - \sqrt{\gamma})^2  -\gamma\right] J(x_2) \geq 
 \left[(1-\eta^2) (1 - \sqrt{\gamma})^2  -\gamma\right] C^{-2}\frac{1}{2} \|\nabla J(x_2)\|^2, 
\end{align*}
with $C = \sup_{x \in \Bx} \|F'(x)^*\|$.
\end{proof}
This lemma justifies our claim that \Nc{$\gamma$,$\beta$}  is  a generalization of the tangential cone condition. 
Note, however, that quasiconvexity (i.e.,  \Nc{1,0}) or even convexity of the least-squares functional 
is not implied by the tangential cone conditions, while by our results, strong convergence  holds for 
quasiconvex (and convex) functionals if the balancing condition is additionally satisfied. 

Concerning the balancing condition, it can be shown that the  least-squares functional is $\gamma$-balanced if 
$F$ satisfies the strong tangential cone condition. 
\begin{lemma}
Let  \eqref{ssc} hold with $\eta <1$. Then  the least-squares 
functional \eqref{least squares} is $\gamma$-balanced around $\xs$ for any $\gamma \in (0,1]$. 
\end{lemma} 
\begin{proof}
Indeed it follows from \eqref{furthermore} that 
\begin{align*}
2 J(\xs -\tau z) &= \|F(\xs -\tau z) - F(\xs) \|^2 \leq 
\frac{2}{1 -\eta} \|F'(\xs) (-\tau z)\|^2 \\
& = 
\tau \frac{2}{1 -\eta} \|F'(\xs) z \|^2 \leq 
\tau \frac{2(1 +\eta)}{1 -\eta} \|F(\xs+ z) - F(\xs)\|^2 \\
&= 
\tau \frac{1 +\eta}{(1 -\eta) \gamma} \gamma 2 J(\xs + z) 
\end{align*}
Thus $\tau = \frac{(1-\eta)\gamma}{(1+\eta)}$ provides \eqref{fancy1}.
\end{proof} 
We note that for convergence of the Landweber iteration, often the tangential cone condition is 
imposed with $\eta < \frac{1}{2}$. A consequence of our results is that strong convergence 
also follows with $\eta <1$.

 We provide another sufficient condition 
for the balancing condition \eqref{fancy1} if the classical weak tangential cone condition holds. 
\begin{lemma}
Let  $\xs$ is the unique global minimum in $\Bx$ and  let the weak tangential cone condition 
 \eqref{wsc} hold for some $\eta <1$.
If for any sequence with $\Delta_n \to 0$
\begin{equation}\label{indeterm} \limsup_{n} \frac{J(\xs + \Delta_n)}{J(\xs -\Delta_N)} > 0, 
\end{equation}
holds, then \eqref{fancy1} is satisfied for any $\gamma >0$.
\end{lemma}
\begin{proof}
Fix $\gamma >0$ and  suppose that  \eqref{fancy1} does not hold. Then we find a sequence $z_n$ with 
 $\rho_0 \leq \|z_n\| \leq \rho$ and a sequence of $\tau_n>0$ with $\tau_n \to_{n \to \infty} 0$  and
 \[ J(\xs - \tau_n z_n ) >  \gamma J(\xs+ z_n). \]
However, by  Proposition~\ref{thisprop} it follows that 
$J(\xs+ z_n) \geq \tau_{n}^{-2(1-\eta)} J(\xs+ \tau_n z_n)$. 
Since $\Delta_n:= \tau_n z_n \to 0$,   
we obtain 
\[ \frac{J(\xs + \Delta_n)}{J(\xs -\Delta_n)} < \frac{ \tau_{n}^{2(1-\eta)} }{\gamma}  \to 0, \]
which contradicts \eqref{indeterm}.
 \end{proof}
This also illustrates that for linear problems, the balancing conditions is trivial. Indeed, as
the functional $J(\xs + \Delta)$ is a quadratic form $(A\Delta,\Delta)$ then, the ratio in \eqref{indeterm}
is  always $1$. If $J$ can be estimated around $\xs$ from below and above by a constant 
times an even-homogeneous functional (similar to \eqref{furthermore}), then \eqref{indeterm} is 
satisfied.

As a justification for our claim of a unification of nonlinearity conditions, we present the 
implications of these conditions in the following scheme: \\[2mm]

{\small
\begin{tabular}{ccc ccc c }
\text{convexity}  &  $\Rightarrow$  &  \text{quasiconvexity} &  $\Rightarrow$ & 
{\Nc{$\gamma$,0}}  &  
  \\[2mm]
 & & & &  $\Downarrow$  &  \\
\mbox{$ \begin{array}{c} 
 \text{weak tangential} \\
 \text{cone cond. \eqref{wsc}}\\
 \end{array}$}
 &$\Rightarrow$   
 & \text{\Nc{0,$\beta <0$}}  
 & 
 $\Rightarrow$ 
 &
 \text{\Nc{0,0}} 
 & $\Rightarrow$ 
 \text{\Nc{0,$\beta>0$}}
 & \\
  & & & &  & $\Downarrow$   \\
    & & & &  &  \mbox{$\begin{array}{c}
    \text{weak} \\
    \text{convergence} 
    \end{array}$} 
\end{tabular}
}\\[2mm]

While most of the traditional
(weak) convergence proofs use the left (separated) assumptions, we employ a 
weaker version (right-hand side)  that includes both of them as special cases.
The main result about weak convergence in this paper is indicated in  the last line of this table.

\section{Weak convergence}\label{Sec:three}

For the following analysis, it is convenient to introduce some shorthand notations
both for the noisy and exact case:  
\begin{align}
e_k &:= x_k - \xs,& \ed_k&:= \xd_k - \xs,&  \label{short1}  \\
\T_k &:= \nabla J(x_k) 
& 
 \Td_k &:= \nabla \Jd(\xd_k), 
 &  
 \label{short2} 
\end{align}
The gradient iterations can then be written as 
\begin{equation}\label{LW} \ed_{k+1} = \ed_{k+1} - \Td_k, \qquad e_{k+1} = e_{k+1} - \T_k. 
\end{equation}
The first lemma concerns monotonicity of the functional values. 
\begin{lemma}\label{lem:monotonicity}
Let  Assumption~\ref{Assmain}  hold and let $x_k, x_{k+1} \in \Bx$ be defined 
by \eqref{lwexact}.
Then the  functional values are monotonically decreasing:
\begin{equation}
J(x_{k+1}) \leq J(x_{k}) \, .
\end{equation}
Moreover, if $x_k \in \Bx$ for $k = 0,\ldots N$, then 
\[  \sum_{k=0}^{N-1} \|x_{k+1}-x_{k}\|^2  = 
\sum_{k=0}^{N-1} \|\T_k\|^2   \leq \frac{1}{ |1-L|} J(x_0)  <  \infty\, . \] 
\end{lemma}
\begin{proof}
By Lipschitz continuity and  Assumption~\ref{Assmain}, we have  using $\Delta_k = x_{k+1}-x_{k}$, 
\begin{equation}\label{help1}  \left|J(x_k+\Delta_k) - J(x_k) -\scl \nabla J(x_k),\Delta_k \scr 
\right| \leq L \|\Delta_k\|^2. 
\end{equation} 
By \eqref{lwexact}  we have 
\[ \scl \nabla J (x_k), \Delta_k  \scr = - \|\T_k\|^2 = 
-\|\Delta_k\|^2. \]
Thus with \eqref{help1} and \eqref{A2k}, we obtain  
\begin{align}\label{help11}
J(x_{k+1})  -  J(x_{k})   =  J(x_k+\Delta_k) - J(x_{k})  \leq 
%
(L -1)  \|\Delta_k\|^2 <0,  
\end{align}
which proves the first assertion.  
A telescope sum,
\[ J(x_N)  -  J(x_0)   + |1-L| \sum_{l=0}^{N-1} \|\T_l\|^2 < 0, \]
 yields the second result.
\end{proof}

By completely the same proof and by replacing $J$ by $\Jd$ and using the ``noisy'' variables instead of the 
exact ones, we can verify the analogous result for $\Jd$.
\begin{lemma}\label{lem:monotonicityd}
Let  Assumption~\ref{Assmain}  hold and 
let $\xd_k $ be defined 
by \eqref{lwexact} and  let $\xd_k, \xd_{k+1} \in \Bx$.  
Moreover, assume for the Lipschitz constant of $\Jd$ that  $L_\delta <1.$ 
Then the  corresponding residuals are monotonically decreasing:
\begin{equation}
\Jd(\xd_{k+1})  \leq  \Jd(\xd_{k}) .
\end{equation}
Moreover, if $\xd_k \in \Bx$ for $k = 0,\ldots N$, then 
\[  \sum_{k=0}^{N-1} \|\xd_{k+1}-\xd_{k}\|^2  = \sum_{k=0}^{N-1} \|\Td_k\|^2 
\leq \frac{1}{ |1-L_\delta|} \Jd(x_0) <  \infty\, . \] 
\end{lemma}

Next, we consider uniform bounds for the error for the iteration with the exact functional $J$. 
We recall the definition of the positive part $ f^+:= \max(f,0) $. 
\begin{lemma}\label{lemma2}
Let  Assumption~\ref{Assmain}  hold. 
Suppose that $x_k \in \Bx,$ for $k = 0,\ldots, N.$
Assume that $\Nc{0,\beta}$ holds for some $\beta \in \R.$ 
Then 
\[ \|e_{k+1}\|^2  \leq \|e_0\|^2 +  \frac{(1 + 2 \beta)^+}{ |1-L|} J(x_0), \qquad 
k = 0,\ldots N.\]
\end{lemma}
\begin{proof}
By \eqref{LW} and with \eqref{weakgamma}, we have for $k \leq N$ 
\begin{align*} 
\|e_{k+1}\|^2 & = \|e_k\|^2 - 2 \scl \T_k, e_k \scr  + \|\T_k\|^2 \leq  \|e_k\|^2 +2 \beta \|\T_k\|^2   + \|\T_k\|^2 \\
& =  \|e_k\|^2  + (1 + 2 \beta) \|\T_k\|^2. 
\end{align*}
By telescoping we find with Lemma~\ref{lem:monotonicity} 
\[ \|e_{k+1}\|^2  - \|e_0\|^2 \leq   (1 + 2 \beta)  \sum_{l=0}^{k-1}  \|\T_l\|^2 \leq 
   \frac{(1 + 2 \beta)^+}{ |1-L|}  J(x_0). \]
\end{proof}

This lemma gives boundedness of the exact Landweber iteration. 
\begin{corollary}\label{corr}
Let  Assumption~\ref{Assmain}  and  let $\Nc{0,\beta}$ hold for some $\beta \in \R.$  Suppose that $x_0$ is such that 
\begin{equation}\label{init}  \|x_0 - \xs\|^2 +  \frac{(1 + 2 \beta)^+}{ |1-L|}   J(x_0) < \rho^2 \, .
\end{equation} 
Then $x_k \in \Bx$ for all $k \geq 0$.
\end{corollary} 
\begin{proof}
We proceed by induction. Clearly $x_0 \in \Bx$ by \eqref{init}. Suppose that $x_l \in \Bx$ for all $0 \leq l \leq k$. 
Then Lemma~\ref{lemma2} with $N = k$ yields that 
$\|e_{k+1} \| < \rho$, thus, $x_{k+1} \in \Bx$. By induction  it follows that $x_k \in \Bx$ for all $k \geq 0$. 
\end{proof} 

Next, we consider the noisy iteration and verify a uniform bound for $\ed_k$. 
The first  lemma provides a recursive estimate. 
\begin{lemma}\label{lem:rec}
Let Assumption~\ref{Assmain} hold. 
Suppose that $\xd_k \in \Bx$ and let $\Nc{0,\beta}$ hold for some $\beta \in \R$. 
Then 
\begin{equation}\label{eq:rec} \|\ed_{k+1}\|^2 
 \leq \|\ed_k\|^2  + \|\Td_k\|^2 \Cm+ 2 \delta  \|\ed_k\|  + 4 \beta^+ \delta^2 .
\end{equation}
with 
\begin{equation}\label{defcm} \Cm  = (1+ 4 \beta^+).  \end{equation}
\end{lemma}
\begin{proof}
Define $\Hd_k = \Td(\xd_k)  - \T(\xd_k)$.
We obtain with the help of  \eqref{weakgamma}, \eqref{five}, and Young's  inequality,
\begin{align*} 
\|\ed_{k+1}\|^2  &= \|\ed_k\|^2  + \|\Td_k\|^2 - 2 \scl \Td_k,\ed_k \scr  \\
&=
 \|\ed_k\|^2  + \|\Td_k\|^2 - 2 \scl \T(\xd_k),\ed_k\scr - 2 \scl\Hd_k,\ed_k \scr     \\
& \leq   
 \|\ed_k\|^2  + \|\Td_k\|^2 + 2 \beta \|\T(\xd_k)\|^2   + 2 \|\ed_k\| \|\Hd_k\| \\
&  =  \|\ed_k\|^2  + \|\Td_k\|^2 + 2 \beta \|\Td_k + \Hd_k \|^2 + 2 \|\ed_k\|  \|\Hd_k\|  \\
&  \leq  \|\ed_k\|^2  + \|\Td_k\|^2 + 2 \beta \|\Td_k\|^2  + 2 \beta \|\Hd_k\|^2  \\
& \qquad \qquad + 
4 |\beta|  \|\Td_k\|   \|\Hd_k\|
+ 2 \|\ed_k\|  \|\Hd_k\|  \\
%
 &  \leq   \|\ed_k\|^2  + \|\Td_k\|^2(1+  2 \beta + 2 |\beta| )  + 
 (2 |\beta| +  2\beta)   \|\Hd_k \|^2 \\
 & \qquad + 2   \|\Hd_k\|  \|\ed_k\|
\end{align*}
With \eqref{five}, the inequality  \eqref{eq:rec} follows .
\end{proof}
The next lemma provides a uniform bound. 
\begin{lemma}\label{dreisechs}
Let Assumption~\ref{Assmain} hold and let $\Ld <1$. Suppose that $\xd_k \in \Bx$ for $k = 0,\ldots, N$,
and  let $\Nc{0,\beta}$  hold for some $\beta \in \R$. 
Define $\fack = \max\{1,2 \sqrt{\beta^+}\}$. 
Then 
 \begin{equation}\label{spliteq1}
\begin{split}
&\|\ed_{k+1}\|^2 + \frac{ \Cm}{|1-\Ld|} \Jd(x_0)  - \sum_{l=0}^{k}\|\Td_k\|^2 \Cm \\
& \leq \left( \sqrt{\|\ed_0\|^2 +  \frac{ \Cm}{ |1-\Ld|} \Jd(x_0)   } + \fack \delta k \right)^2  \qquad k = 0,\ldots N.
\end{split}
\end{equation} 
\end{lemma}
\begin{proof}
We proceed by induction over $N$. 
Let  $N = 0$ and assume that $x_0 \in \Bx$. 
For $k = 0$ we have by \eqref{eq:rec}
\begin{align*}
&\|\ed_{1}\|^2  + \frac{ \Cm}{|1-\Ld|} \Jd(x_0)  - \|\Td_0\|^2 \Cm \\
& \leq \|\ed_0\|^2  + \|\Td_0\|^2 \Cm+ 2 \delta  \|\ed_0\|  + 4 \beta^+ \delta^2  + \frac{ \Cm}{|1-\Ld|} \Jd(x_0)  - \|\Td_0\|^2 \Cm\\
& \leq \|\ed_0\|^2  + 2 \delta  \|\ed_0\|  + 4 \beta^+ \delta^2  + \frac{ \Cm}{|1-\Ld|} \Jd(x_0) \\ 
& \leq \left(\sqrt{ \|\ed_0\|^2  + \frac{ \Cm}{|1-\Ld|} \Jd(x_0) }\right)^2 +  2 \fack \delta  \|\ed_0\|+ 4 \beta^+ \delta^2  \\
& \leq   \left(\sqrt{ \|\ed_0\|^2  + \frac{ \Cm}{|1-\Ld|} \Jd(x_0) } + \xi \delta \right)^2 + \delta^2 (4 \beta^+  -\fack^2).
\end{align*}
Since the last term is negative by definition of $\fack$, the estimate holds for $k = 0 = N$.

Now suppose that if  $\xd_k \in \Bx$ for $k = 0,\ldots, N-1$,  then the estimate \eqref{spliteq1} holds for  for $k =0 \ldots, N-1$.
We show that this is also the case when $N$ is replaced by $N+1$. Thus, let $\xd_k \in \Bx$ for $k = 0,\ldots, N$.
By the induction hypothesis we only have to show that \eqref{spliteq1} holds for $k = N$. 

By Lemma~\ref{lem:monotonicityd},  we obtain
\begin{equation}\label{this}  \sum_{l=0}^{N-1} \|\Td_l\|^2 \leq  \frac{1}{ |1-\Ld|} \Jd(x_0) .
\end{equation}
For brevity, define $\myeta =  \frac{ \Cm}{ |1-\Ld|} \Jd(x_0).$ 
By Lemma~\ref{lem:rec} and since $\fack \geq 1$, we find
\begin{align*} 
\|\ed_{N+1}\|^2  -  \Cm\sum_{l=0}^{N} \|\Td_k\|^2 + \myeta &\leq 
\|\ed_{N}\|^2  - \Cm \sum_{l=0}^{N-1} \|\Td_k\|^2 + \myeta  + 2 \delta \|\ed_{N-1}\| + 4 \beta^+ \delta^2\\
& \leq \|\ed_{N}\|^2  - \Cm \sum_{l=0}^{N-1} \|\Td_k\|^2 + \myeta  + 2 \fack \delta \|\ed_{N}\| + 4 \beta^+ \delta^2.
\end{align*}
According to the induction hypothesis we have \eqref{spliteq1}, which allows to 
estimate the first three terms on the right-hand side. Moreover, by \eqref{this} and  \eqref{spliteq1},
again $\|\ed_{N}\|$ can be bounded by the right-hand side in \eqref{spliteq1}. Thus 
\begin{align*} 
&\|\ed_{N+1}\|^2  -  \Cm\sum_{l=0}^{N} \|\Td_k\|^2 + \myeta 
\leq 
 \left( \sqrt{\|\ed_0\|^2 +  \frac{ \Cm}{ |1-\Ld|}  \Jd(x_0)} + \fack \delta (N-1) \right)^2  
\\
& \qquad \qquad + 2 \delta \fack 
  \left( \sqrt{\|\ed_0\|^2 +  \frac{ \Cm}{ |1-\Ld|}  \Jd(x_0)} + \fack \delta (N-1) \right)
  +4 \beta^+ \delta^2,
\end{align*}
By completing the square as before and since $(4 \beta^+  -\fack^2) \leq 0$ we find \eqref{spliteq1} for $k \leq N,$
which proves the lemma.
\end{proof}

%

We have the following proposition: 
\begin{proposition}\label{prop}
Let Assumption~\ref{Assmain} and $\Nc{0,\beta}$ hold for some $\beta \in \R$.  
Let $L_\delta <1$ in $\Bx$ and  $\xd_0$ and $N\geq 0$ be such that 
\begin{equation}\label{initbound} 
      \left( \sqrt{\|e_0\|^2 +  \frac{ \Cm}{ |1-\Ld|} \Jd(x_0)  } + \fack \delta N \right)^2
   +  \Cm \FL(\Jd(x_0)) \leq \rho^2, 
  \end{equation}
 where $\Cm$ is defined in \eqref{defcm}. 
Then  for all $k \leq N,$  the sequence $\xd_k$ is in $\Bx$  and we have the estimate 
\eqref{spliteq1} for $k = 0,\ldots, N$. 
\end{proposition} 
\begin{proof}
We use induction over $k\leq N$. For $k = 0$, $\xd_0$ is in $\Bx$  by \eqref{initbound}. 
Let $\xd_l \in \Bx$ for $l = 0,\ldots k,$ $k < N$. We show that  $\xd_{k+1} \in \Bx$. 

From \eqref{this} (with the sum up to the index $k-1$) and \eqref{spliteq1} we may estimate 
\begin{align*} 
&\|\ed_{k+1}\|^2 \leq 
 \left( \sqrt{\|\ed_0\|^2 +  \frac{ \Cm}{|1-\Ld|} \Jd(x_0) } + \fack \delta k \right)^2  
  +  \Cm \|\Td_k\|^2.  
\end{align*}
Using \eqref{Fleq} for the last term on the right-hand side and  by Lemma~\ref{lem:monotonicityd} and
$\Jd(x_k) \leq \Jd(\xd_0),$ we observe that $\|\ed_{k+1}\| \leq \rho,$ thus $x_{k+1} \in \Bx$. 
Induction yields the assertion. The estimate \eqref{spliteq1} follows from Lemma~\ref{dreisechs}
%
\end{proof} 

Since it is well-known that the Landweber iteration has to be stopped 
for noisy data, we have to introduce a stopping criterion. 
Here we choose a simple a-priori rule: 
for each noise level $\delta$ define the stopping index $\Nd$ such that 

\begin{equation}\label{bla}
\lim_{\delta \to 0} \Nd = \infty, \qquad \lim_{\delta \to 0} \Nd \delta = 0, \qquad 
(\Nd+1) \delta \leq \frac{\rho}{2 \fack}. 
\end{equation}
We have the following theorem. 
\begin{theorem}\label{mainweak}
Let Assumption~\ref{Assmain} and $\Nc{0,\beta}$
hold for some $\beta \in \R$.
Let $x_0$ be close to $\xs$ such that 
\begin{equation}\label{thisinw}
\|e_0\|^2 +  \frac{2 \Cm}{ |1-L|} J(x_0)
   +  \Cm \FL (J(x_0))  \leq \frac{1}{16} \rho^2.
  \end{equation}
Let $\delta_l$ be a sequence of noise levels associated to 
noisy data via \eqref{five} and let them 
be sufficiently small such that 
\begin{equation}\label{suchthat}
\delta_l <  \frac{1-L}{2}, \qquad  
\FL(J(x_0) + \psi(\delta)) \leq  \FL(J(x_0))  + \frac{\rho^2}{8\Cm }, \qquad 
\frac{2 \Cm}{ |1-L|}  \psi(\delta)  \leq  \frac{\rho^2}{8}, 
\end{equation}
%
holds, and let the stopping index be chosen as in \eqref{bla}. 
Then $\xdl_{\Ndl}$ is in $\Bx$ and hence has a weakly convergent subsequence.

If  $x\to \T(x)$ is weakly sequentially closed on $\Bx$, 
then a limit of this subsequence is a stationary point of $J$.
Assume additionally that $\xs$ is the unique stationary point of $J$ in $\Bx$.
Then 
\[  \xdl_{\Ndl} \rightharpoonup \xs,  \qquad \text{as } \delta_l \to 0.\]
\end{theorem}
\begin{proof}
Since $\delta_l$ is small, from \eqref{noise1} if follows that $\Jd$ is Lipschitz with 
$\Ld <1$ and $\frac{1}{1-\Ld} \leq \frac{2}{1-L}$ for all $\delta = \delta_l$.
With \eqref{bla} and  \eqref{thisinw}, it may be verified that \eqref{initbound} holds 
for all $\delta_l$  and with $N = \Nd+1$.  
Thus, by Proposition~\ref{prop}, the iterates $\xd_k \in \Bx$ for all indices $k$ up
to the stopping index $\Nd+1$. In particular, $\xdl_{\Ndl}$ is bounded and has a weakly convergence
subsequence. 

Since $\xdl_{\Ndl+1} \in \Bx$, we have by \eqref{help11} that 
\[ \|\Td_{\Nd}\|^2 \leq  \frac{2}{1-L} \left(\Jd_{\Ndl} -  \Jd_{\Ndl-1} \right). \] 
From \eqref{Jfive} we find that 
\[ \|\Td_{\Nd}\|^2 \leq  \frac{2}{1-L} \left|J_{\Ndl} -  J_{\Ndl-1} \right| +  \frac{2}{1-L}\psi(\delta_l).  \] 
By Corollary~\ref{corr}, the sequence $x_k$ is in $\Bx$, hence by Lemma~\ref{lem:monotonicity}, 
the sequence $J_k$ is decreasing and hence convergent. Since $\Ndl \to \infty$, and $\delta_l \to 0$, we 
conclude by \eqref{five}  that
\[ \lim_{\delta_l\to 0} \|\T_{\Nd}\| = 0. \] 
Then, by weakly closedness,
\[ \xdl \rightharpoonup \tilde{x} \quad \mbox{ and }  \T(\xdl) \to 0 \Rightarrow \T(\tilde{x}) = 0. \] 
Hence the limit point is a stationary point. If the stationary point in $\Bx$ is unique, 
the any subsequence  has a weakly convergent subsequence with limit $\xs$, thus
$\xd_{\Nd}$ must converge weakly to $\xs$. 
\end{proof}

\section{Strong convergence}\label{Sec:four}
The next step in the analysis concerns a proof of strong convergence of the iterations. 
As it could be expected, this requires additional conditions, namely the functional has to 
be $\gamma$ balanced and satisfies \Nc{$\gamma$,$\beta$}.

\begin{lemma}\label{con}
Let Assumption~\ref{Assmain} and   \Nc{$0$,$\beta$} hold for some $\beta \in \R$. Let $x_0-\xs$ small enough that 
\eqref{init} holds. 
Then 
\begin{equation}\label{myxxx} \sum_{k=1}^\infty |\scl \T_k,e_k \scr|  < \infty. \end{equation}
\end{lemma}
\begin{proof}
From \eqref{LW} we obtain that for any integer $n_1< n_2$ that  
\begin{align*}  
&\sum_{k=n_1}^{n_2} \left[\scl \T_k,e_k \scr+ (\T_k, e_{k+1}) \right] = 
- \sum_{k=n_1}^{n_2} \left[(e_{k+1} -e_{k},e_k+ e_{k+1}) \right] \\ 
& = -  \sum_{k=n_1}^{n_2}  \left[\|e_{k+1}\|^2 - \|e_{k}\|^2\right]  = - \|e_{n_2+1}\|^2 +\| e_{n_1}\|^2. 
\end{align*}
Moreover,  from  $(\T_k, e_{k+1}) = (\T_k, e_{k}) - \|\T_k\|^2$  we obtain 
\begin{align*}
&2 \sum_{k=n_1}^{n_2} \scl \T_k,e_k \scr = - \|e_{n_2+1}\|^2 +\| e_{n_1}\|^2  + \sum_{k=n_1}^{n_2} \|\T_k\|^2 \, .
\end{align*}
 We split  the sum into $I_1 = \{k \in [n_1,n_2]\,|   \scl \T_k,e_k \scr\geq 0\}$ 
 and $I_2 = \{k \in [n_1,n_2]\,|   \scl \T_k,e_k \scr < 0\}$ and use \eqref{weakgamma} to find 
\begin{align*}  
2 \sum_{k=n_1}^{n_2} |\scl \T_k,e_k \scr|   
&= 2 \sum_{k \in I_1}  \scl \T_k,e_k \scr   - 2 \sum_{k \in I_2} \scl \T_k,e_k \scr  \\
& =  - \|e_{n_2+1}\|^2 +\| e_{n_1}\|^2  + \sum_{k=n_1}^{n_2} \|\T_k\|^2  - 4 \sum_{k \in I_2} \scl \T_k,e_k \scr  \\
& \leq  - \|e_{n_2+1}\|^2 +\| e_{n_1}\|^2  + (1 + 4 \beta^+ )  \sum_{k = n_1}^{n_2} \|\T_k\|^2. 
\end{align*}
According to Lemma~\ref{lemma2} and \ref{lem:monotonicity}, the right-hand side is uniformly bounded.
\end{proof}

\begin{lemma}\label{ex3}
Let Assumption~\ref{Assmain} hold.
Suppose that $J$ satisfies \Nc{$\gamma$,$\beta$} and 
is  $\gamma$-balanced for some $\gamma  \geq 0$ and for some $\beta \in \R$.
For a subsequence $\esub{m}$  assume that 
$\liminf_m \|\esub{m}\| \geq \co >0.$ 
Then for any $0\leq s < k_m$ and $m \geq m_0$ 
\[ - \scl \T_s,\esub{m} \scr    \leq  \frac{1}{\tau}  \scl \T_s,e_s \scr  +  \frac{\beta}{\tau} \| \nabla J(x_s)\|^2 \, . \]
\end{lemma}
\begin{proof}
By \eqref{fancy1}, we find   a $\tau>0$ with
\[ J(\xs -\tau \esub{m}) \leq J(\xs +  \esub{m}) = \gamma J(x_{k_m}) \qquad \forall m \geq m_0 \]
Thus if $s \leq  k_m,$ and $m \geq m_0,$ we have by Lemma~\ref{lem:monotonicity} that $J(\xs -\tau \esub{m}) \leq  \gamma J(x_{k_m}) \leq  
\gamma J(x_s)$. 
We apply \eqref{weakgamma} with $x_2 = x_s$ and $x_1 =   \xs -\tau \esub{m}$.
It then holds that $x_1, x_2 \in \Bx$. This yields 
\[ \beta \| \nabla J(x_s)\|^2  \geq \scl \nabla J(x_s),\xs -\tau \esub{m} -x_s \scr  = -\scl \T_s, e_s\scr - \tau \scl \T_s,\esub{m} \scr. \]
%
\end{proof}

Summing up we arrive at the following theorem on convergence of the exact iteration.
\begin{theorem}\label{th:bbb}
Let Assumption~\ref{Assmain}  hold and 
suppose that $J$ satisfies \Nc{$\gamma$,$\beta$} and 
is  $\gamma$-balanced for some $\gamma  \geq 0$ and for some $\beta \in \R$.
Let $x_0-\xs$ small enough such that 
\eqref{init} holds. 
Then $x_k$ has a strongly convergent subsequence with a limit that is a stationary point. 
If $\xs$ is the unique stationary point of $J$ in $\Bx$, then the sequence $x_k$ converges to $\xs$. 
\end{theorem}
\begin{proof}
%
By Corollary~\ref{corr}, $e_k$ is bounded for all $k.$ 
Hence, there exist a subsequence, where $\|\esub{m}\|$ is convergent. 
If $e_k$ has a strongly convergent subsequence  with limit $0$, then  we are finished. 
Otherwise,  for any subsequence we have $\liminf \|\esub{m}\| \geq \co >0,$ in particular 
also for one for which $ \|\esub{m}\|$ is convergent.  
Take such a subsequence and write for $k_m>k_n \geq m_0,$ where $m_0$ is the index 
in \eqref{fancy1} 
\begin{align*}  \|\esub{n} - \esub{m}\|^2  &=  \|\esub{n}\|^2 - \|\esub{m}\|^2 + 2\scl\esub{m}-\esub{n},\esub{m}\scr \\
& = \|\esub{n}\|^2 - \|\esub{m}\|^2 -  2\scl\sum_{s=k_n}^{k_m-1} \T_s ,\esub{m}\scr \, .
\end{align*}
From Lemma~\ref{ex3}, we obtain that 
\begin{align*}  \|\esub{n} - \esub{m}\|^2  
& \leq  \|\esub{n}\|^2 - \|\esub{m}\|^2 + \frac{2}{\tau} \sum_{s=k_n}^{k_m-1} \scl \T_s ,e_s\scr + \frac{2 \beta}{\tau}\sum_{s=k_n}^{k_m-1} \| \nabla J(x_s)\|^2 \\
&\leq 
\|\esub{n}\|^2 - \|\esub{m}\|^2 + \frac{2}{\tau} \sum_{s=k_n}^{k_m-1} |\scl \T_s ,e_s\scr |   +
\frac{2\beta}{\tau}\sum_{s=k_n}^{k_m-1} \| \nabla J(x_s)\|^2 \,.
\end{align*}
By  \eqref{myxxx} 
and since $\|\esub{m}\|$ is convergent, we may find for any given $\epsilon$ an $n_0 \geq m_0$ such that 
for all $k_m>k_n > n_0$ the right-hand side is smaller than  $\epsilon$. 
Thus, $\esub{n}$ is a Cauchy sequence and hence convergent. Since by Lemma~\ref{lem:monotonicity}, 
$\T_{k_m} \to 0,$ and 
$\T$ is continuous, it follows that the limit $\tilde{x}$ must be a stationary point. 
If $\xs$ is the only possibility of such a limit, it follows by a standard subsequence argument, 
that $x_k$ must converge to $\xs$. 
\end{proof}

We now come to the main result of strong convergence in the noisy case. 
Concerning the stopping criterion, we define for 
each noise level $\delta$ the stopping index $\Nd$ according to \eqref{bla}.
%
Then we have the following theorem. 
\begin{theorem}\label{th:mainstrong}
Let Assumption~\ref{Assmain}  hold and 
suppose that $J$ satisfies \Nc{$\gamma$,$\beta$} and 
is  $\gamma$-balanced for some $\gamma  \geq 0$ and for some $\beta \in \R$.
Let $x_0-\xs$ small enough such that 
\eqref{init} holds. 
Let  the  sequence of noise levels  $\delta_l \to 0$  be sufficiently small such that \eqref{suchthat} holds and 
let the stopping index be chosen as in \eqref{bla}. 

Assume that $\xs$ is the unique stationary point of $J$ in $\Bx$.
Then 
\[ \lim_{\delta \to 0} \xd_{\Nd} = \xs, \]
\end{theorem}
\begin{proof}
As in Theorem~\ref{mainweak}, $\ed_k$ for $k = 0,\ldots \Nd+1$ is bounded by $\rho$
such that  $\xd_{\Nd}$ is always in $\Bx$ and hence uniformly bounded. 
Take a fixed $m$  and assume that $\delta$ is sufficiently small such that 
$\Nd \geq m$. 
%
From \eqref{eq:rec} we may estimate recursively that (using the fact that we may take 
$\beta = 0$)
%
%
\begin{align*}  \|\xd_{\Nd} -\xs\|^2  &= \|\ed_{\Nd}\|^2 \leq 
\|\ed_{m}\|^2 +  \sum_{k=m}^{\Nd-1} \Cm \|\Td_k\|^2 + 
2 \delta \sum_{k=m}^{\Nd-1} \|\ed_k\| + 4 \beta^+\delta^2 (\Nd -m) \\
& \leq 
\|\ed_{m}\|^2 +  \sum_{k=m}^{\Nd-1} \Cm \|\Td_k\|^2 + 2 \delta \Nd \rho + 4 \beta^+\delta^2 \Nd \\
&\leq 
\|\ed_{m}\|^2 +  \sum_{k=m}^{\Nd-1}\Cm  \|\T_k\|^2 + 2 \Cm \Nd \delta^2 + 2 \delta \Nd \rho  + 4 \beta^+\delta^2 \Nd, 
%
%
\end{align*}
where we used \eqref{five} in the last step.
The recursion for $\xd_k - x_k$ might be estimated by Assumption~\ref{Assmain}:
\[ \|\xd_{k+1} - x_{k+1}\| \leq (1 + L) \|\xd_{k} - x_{k}\|  + \delta  \Rightarrow  
\|\xd_{k+1} - x_{k+1}\|  \leq \delta \frac{1}{L}\left( (1+ L)^{k+1} -1 \right). \]
Thus for $\Nd \geq m$ we have 
\begin{align*}  &\|\xd_{\Nd} -\xs\|^2  \\
&\leq 
2 \|e_{m}\|^2  + 2\delta^2  \tfrac{1}{L^2}\left( (1+ L)^{m} -1 \right)^2 \\
& \qquad +  
\sum_{k=m}^{\Nd-1} \Cm  \|\T_k\|^2 + 2 \Nd \Cm \delta^2 + 2 \delta \Nd \rho   + 4 \beta^+\delta^2 \Nd. 
%
%
\end{align*}
Fix an $\epsilon>0.$ Since  the assumptions imply that $\|e_m\|$ converges to $0$ and
the sum of the squares of the gradients is convergent, and by the parameter choice \eqref{bla}, 
we may  find a $m$ (depending on $\epsilon$)  and a $\delta_0$ such that for all $\delta \leq \delta_0$ 
\[ \|\xd_{\Nd} -\xs\|^2  \leq 2\delta^2 \tfrac{1}{L^2}\left( (1+ L)^{m} -1 \right)^2 + \epsilon. \]
Taking $\delta$ even smaller (depending on $m$) yields that the right-hand side is 
smaller than $2\epsilon$. Thus $\lim_{\delta \to 0} \xd_{\Nd} = \xs$. 
\end{proof} 

As a corollary we obtain a result which cannot be proven by the approach via tangential cone conditions.
\begin{corollary}\label{th:convex}
Let Assumption~\ref{Assmain}  hold  and let the level sets $\{J < \alpha\}\cap \Bx$
be convex and let $J$ be $1$-balanced. With  
$x_0-\xs$, $\delta$ and the stopping index as in Theorem~\ref{th:mainstrong}
for $\gamma = 1$, and if $\xs$ is the unique stationary point of $J$ in $\Bx$, then
we have
\[ \lim_{\delta \to 0} \xd_{\Nd} = \xs, \]
\end{corollary}

\section{Conclusion}
In this paper, we considered gradient descent iterations for functionals 
with Lipschitz-continuous derivative. We introduced new restrictions on the 
nonlinearity of the problem, namely the conditions \Nc{$\gamma$,$\beta$} 
and the $\gamma$-balancing conditions. We have shown that they are weaker than
several classical nonlinearity conditions.

The first main result concern weak convergence for the exact and noisy 
case of gradient iterations if the condition \Nc{0,$\beta$} with some $\beta \in \R$ 
holds and using an a-priori stopping rule. 

Strong convergence is verified in the exact 
case if \Nc{$\gamma$,$\beta$}, $\beta \in \R$, $\gamma \in [0,1]$, holds and the functional is $\gamma$-balanced.  
With a stopping rule and if $\xs$ is the unique global minimum, then strong 
convergence in the noisy case is verified under the same conditions.

\bibliographystyle{siam}
\bibliography{main}

\begin{thebibliography}{10}

\bibitem{arrent}
{\sc K.~J. Arrow and A.~C. Enthoven}, {\em Quasi-concave programming},
  Econometrica, 29 (1961), pp.~779--800.

\bibitem{genco}
{\sc M.~Avriel, W.~E. Diewert, S.~Schaible, and I.~Zang}, {\em Generalized
  concavity}, vol.~63 of Classics in Applied Mathematics, Society for
  Industrial and Applied Mathematics (SIAM), Philadelphia, PA, 2010.

\bibitem{BaGo94}
{\sc A.~Bakushinsky and A.~Goncharsky}, {\em Ill-posed problems: theory and
  applications}, vol.~301 of Mathematics and its Applications, Kluwer Academic
  Publishers Group, Dordrecht, 1994.

\bibitem{BaSm06}
{\sc A.~Bakushinsky and A.~Smirnova}, {\em A posteriori stopping rule for
  regularized fixed point iterations}, Nonlinear Anal., 64 (2006),
  pp.~1255--1261.

\bibitem{BaSm07}
\leavevmode\vrule height 2pt depth -1.6pt width 23pt, {\em Iterative
  regularization and generalized discrepancy principle for monotone operator
  equations}, Numer. Funct. Anal. Optim., 28 (2007), pp.~13--25.

\bibitem{BaKuSm11}
{\sc A.~B. Bakushinsky, M.~Y. Kokurin, and A.~Smirnova}, {\em Iterative methods
  for ill-posed problems}, vol.~54 of Inverse and Ill-posed Problems Series,
  Walter de Gruyter GmbH \& Co. KG, Berlin, 2011.

\bibitem{Cal80}
{\sc A.-P. Calder{\'o}n}, {\em On an inverse boundary value problem}, in
  Seminar on {N}umerical {A}nalysis and its {A}pplications to {C}ontinuum
  {P}hysics ({R}io de {J}aneiro, 1980), Soc. Brasil. Mat., Rio de Janeiro,
  1980, pp.~65--73.

\bibitem{EHN96}
{\sc H.~W. Engl, M.~Hanke, and A.~Neubauer}, {\em Regularization of inverse
  problems}, vol.~375 of Mathematics and its Applications, Kluwer, Dordrecht,
  1996.

\bibitem{HaLeSc07}
{\sc M.~Haltmeier, A.~Leit{\~a}o, and O.~Scherzer}, {\em Kaczmarz methods for
  regularizing nonlinear ill-posed equations. {I}. {C}onvergence analysis},
  Inverse Probl. Imaging, 1 (2007), pp.~289--298.

\bibitem{Ha97}
{\sc M.~Hanke}, {\em A regularizing {L}evenberg-{M}arquardt scheme, with
  applications to inverse groundwater filtration problems}, Inverse Problems,
  13 (1997), pp.~79--95.

\bibitem{HaNeSc95}
{\sc M.~Hanke, A.~Neubauer, and O.~Scherzer}, {\em A convergence analysis of
  the {L}andweber iteration for nonlinear ill-posed problems}, Numer. Math., 72
  (1995), pp.~21--37.

\bibitem{HeKa10}
{\sc T.~Hein and K.~S. Kazimierski}, {\em Accelerated {L}andweber iteration in
  {B}anach spaces}, Inverse Problems, 26 (2010), pp.~055002, 17.

\bibitem{HoRa09}
{\sc N.~S. Hoang and A.~G. Ramm}, {\em Dynamical systems gradient method for
  solving nonlinear equations with monotone operators}, Acta Appl. Math., 106
  (2009), pp.~473--499.

\bibitem{HoRa10}
\leavevmode\vrule height 2pt depth -1.6pt width 23pt, {\em The dynamical
  systems method for solving nonlinear equations with monotone operators},
  Asian-Eur. J. Math., 3 (2010), pp.~57--105.

\bibitem{Jin11}
{\sc Q.~Jin}, {\em A general convergence analysis of some {N}ewton-type methods
  for nonlinear inverse problems}, SIAM J. Numer. Anal., 49 (2011),
  pp.~549--573.

\bibitem{Ka97}
{\sc B.~Kaltenbacher}, {\em Some {N}ewton-type methods for the regularization
  of nonlinear ill-posed problems}, Inverse Problems, 13 (1997), pp.~729--753.

\bibitem{KaNeSc08}
{\sc B.~Kaltenbacher, A.~Neubauer, and O.~Scherzer}, {\em Iterative
  regularization methods for nonlinear ill-posed problems}, vol.~6 of Radon
  Series on Computational and Applied Mathematics, Walter de Gruyter GmbH \&
  Co. KG, Berlin, 2008.

\bibitem{Ne00}
{\sc A.~Neubauer}, {\em On {L}andweber iteration for nonlinear ill-posed
  problems in {H}ilbert scales}, Numer. Math., 85 (2000), pp.~309--328.

\bibitem{Ne15}
\leavevmode\vrule height 2pt depth -1.6pt width 23pt, {\em Some generalizations
  for landweber iteration for nonlinear ill-posed problems}, Journal Inverse
  and Ill-posed Problems,  (2015).
\newblock in press.

\bibitem{PePiSe06}
{\sc S.~S. Pereverzyev, R.~Pinnau, and N.~Siedow}, {\em Regularized fixed-point
  iterations for nonlinear inverse problems}, Inverse Problems, 22 (2006),
  pp.~1--22.

\bibitem{Ra05}
{\sc A.~G. Ramm}, {\em Dynamical systems method for ill-posed equations with
  monotone operators}, Commun. Nonlinear Sci. Numer. Simul., 10 (2005),
  pp.~935--940.

\bibitem{Sch95}
{\sc O.~Scherzer}, {\em Convergence criteria of iterative methods based on
  {L}andweber iteration for solving nonlinear problems}, J. Math. Anal. Appl.,
  194 (1995), pp.~911--933.

\bibitem{ScKaHo12}
{\sc T.~Schuster, B.~Kaltenbacher, B.~Hofmann, and K.~S. Kazimierski}, {\em
  Regularization methods in {B}anach spaces}, vol.~10 of Radon Series on
  Computational and Applied Mathematics, Walter de Gruyter GmbH \& Co. KG,
  Berlin, 2012.

\bibitem{Va98}
{\sc V.~V. Vasin}, {\em Convergence of gradient-type methods for nonlinear
  equations}, Dokl. Akad. Nauk, 359 (1998), pp.~7--9.

\end{thebibliography}

\end{document}